\documentclass{amsart}

\usepackage{amssymb}

\usepackage{amsfonts}
\usepackage{amsmath}
\usepackage{mathtools}
\usepackage{amsthm}

\newtheorem{thm}{Theorem}[section]

\newtheorem{lem}[thm]{Lemma}

\theoremstyle{definition}

\theoremstyle{remark}

\numberwithin{equation}{section}

\newcommand{\R}{\mathbf{R}}  

\DeclareMathOperator{\liet}{\mathfrak{t}}

\DeclareMathOperator{\cO}{\mathcal{O}}

\DeclareMathOperator{\cI}{\mathcal{I}}

\DeclareMathOperator{\cH}{\mathcal{H}}
\DeclareMathOperator{\cD}{\mathcal{D}}
\DeclareMathOperator{\cL}{\mathcal{L}}

\DeclareMathOperator{\bC}{\mathbb{C}}
\DeclareMathOperator{\bG}{\mathbb{G}}

\makeatletter
\newcommand{\colim@}[2]{%
  \vtop{\m@th\ialign{##\cr
    \hfil$#1\operator@font colim$\hfil\cr
    \noalign{\nointerlineskip\kern1.5\ex@}#2\cr
    \noalign{\nointerlineskip\kern-\ex@}\cr}}%
}
\newcommand{\colim}{%
  \mathop{\mathpalette\colim@{\rightarrowfill@\textstyle}}\nmlimits@
}
\makeatother

\makeatletter
\newcommand{\nlim@}[2]{%
  \vtop{\m@th\ialign{##\cr
    \hfil$#1\operator@font lim$\hfil\cr
    \noalign{\nointerlineskip\kern1.5\ex@}#2\cr
    \noalign{\nointerlineskip\kern-\ex@}\cr}}%
}
\newcommand{\nlim}{%
  \mathop{\mathpalette\nlim@{\leftarrowfill@\textstyle}}\nmlimits@
}
\makeatother

\begin{document}

\title{A remark on descent for Coxeter groups}

\author{Gus Lonergan}
\address{Department of Mathematics, Massachusetts Institute of Technology, 
Cambridge, MA 02139}
\email{gusl@mit.edu}


\begin{abstract}
Let $\Gamma$ be a finite Coxeter group with reflection representation $R$. We show that a $\Gamma$-equivariant quasicoherent sheaf on $R$ descends to the quotient space $R//\Gamma$ if it descends to the quotient space $R//\langle s_i\rangle$ for every simple reflection $s_i\in \Gamma$.
\end{abstract}

 \maketitle
\setcounter{tocdepth}{1}
\tableofcontents




\section{Introduction}

\subsection{}Let $W$ be a finite Weyl group with reflection representation $\liet$ over $\bC$, and let $\widetilde{W}$ denote the extended affine Weyl group which also acts on $\liet$ in the natural manner. In \cite[Thm 1.2]{lonerganfourier}, the author has demonstrated that there is an equivalence of categories between modules for a certain algebra, denoted $H$, and the full subcategory of the category of $\widetilde{W}$-equivariant quasicoherent sheaves on $\liet$ which descend to $\liet//\Gamma$ for every finite parabolic subgroup $\Gamma$ of the affine Weyl group $W^{aff}$. On the other hand, it follows directly from the result of V. Ginzburg, \cite[Prop 6.2.5]{ginzburghecke}, that the category of $H$-modules is equivalent to the full subcategory of $\widetilde{W}$-equivariant quasicoherent sheaves on $\liet$ which descend to $\liet//W$. We would like to explain this apparent discrepancy.

\subsection{}In \cite[Lem 2.1.1]{bezmirsom}, it is shown that every reflection in $W^{aff}$ is conjugate in $\widetilde{W}$ to a simple reflection of $W$ \footnote{This fact may be standard but this is the only reference the author is aware of.}. One sees therefore that a $\widetilde{W}$-equivariant quasicoherent sheaf on $\liet$ which descends to $\liet//W$ descends also to $\liet//\langle s\rangle$ for every reflection $s$ in $W^{aff}$. Now fix a finite parabolic subgroup $\Gamma$ of $W^{aff}$. Then $\Gamma$ is generated by its reflections, and is the stabilizer in $W^{aff}$ of some point of $\liet$. Translating to $0$, one sees that $\Gamma$ is the Weyl group of some root subsystem of the root system of $W$. It follows that $\liet$ is the direct sum of the reflection representation $\liet_\Gamma$ of $\Gamma$ and the invariant subspace $(\liet)^\Gamma$. Thus Ginzburg's claim is seen to follow from the author's in light of the following:

\begin{thm}\label{thm1}Let $\Gamma$ be a finite Coxeter group with reflection representation $R$ over $\bC$. Let $X$ be a scheme over $\bC$ and let $Y$ be a $\Gamma$-equivariant $R$-bundle over $X$. Let $M$ be a $\Gamma$-equivariant quasicoherent sheaf on $Y$. Then $M$ descends to $Y//\Gamma$ if it descends to $Y//\langle s_i\rangle$ for every simple reflection $s_i$ in $\Gamma$.
\end{thm}
\noindent We emphasize that the content of this theorem is essentially all contained in the case $Y=R$.

\subsection{Bad grammar.} Let $\Gamma$ be a finite group acting on a scheme $Y$ with GIT quotient $q:Y\to Y//\Gamma$. Suppose $M$ is a quasicoherent sheaf on $Y$. We say \emph{$M$ descends to $Y//\Gamma$} to mean that $M$ is equipped with an isomorphism $M\cong q^*M'$ for some quasicoherent sheaf $M'$ on $Y//\Gamma$. This is not a property of $M$, but rather additional data. Note that in that case $M$ receives a $\Gamma$-equivariant structure. Now suppose instead that $M$ is a $\Gamma$-equivariant quasicoherent sheaf on $Y$. We employ the same phrase: \emph{$M$ descends to $Y//\Gamma$} to mean that the underlying quasicoherent sheaf descends to $Y//\Gamma$ (in the previous sense), and moreover that the induced $W$-equivariant structure coincides with the original one. More generally, for a subgroup $\Gamma'\subset\Gamma$, we will say that \emph{$M$ descends to $Y//\Gamma'$} to mean merely that the underlying $\Gamma'$-equivariant quasicoherent sheaf descends to $Y//\Gamma'$. 

\section{Two Algebras}

\subsection{The nil Hecke algebra.}Let $\Gamma$ be a finite Coxeter group with simple reflections $\{s_i\}_{i\in\Sigma}$. Denote by $\alpha_i$ the corresponding simple roots; then $\cO(R)$ is naturally identified with the symmetric algebra on the various $\alpha_i$. For $i,j\in\Sigma$ let $m_{i,j}$ denote the order of $s_is_j$. Following \cite{kumar}, the nil Hecke algebra is the algebra generated by the symbols $\{D_i\}_{i\in\Sigma}$ subject to the relations
$$
D_i^2=0
$$
and
$$
\underbrace{D_iD_jD_i\ldots}_{m_{i,j}}=\underbrace{D_jD_iD_j\ldots}_{m_{i,j}}
$$
for all $i,j\in\Sigma$. The latter relations imply that for any $w\in\Gamma$ the element 
$$
D_w:=D_{i_1}D_{i_2}\ldots D_{i_l}
$$
is independent of choice of reduced expression $w=s_{i_1}s_{i_2}\ldots s_{i_l}$. Together with the former relations, one sees that $\{D_w\}_{w\in \Gamma}$ form a basis for the nil Hecke algebra, and also that the nil Hecke algebra is graded with each $D_i$ in degree $-1$ by convention.

\subsection{}This algebra acts on $\cO(R)$ by setting
$$
D_i(f)=\alpha_i^{-1}.(1-s_i)(f),
$$
for any $f\in\cO(R)$; this is the so-called Demazure operator. $D_i$ acts by $s_i$-derivations, that is:
$$
D_i(f.g)=D_i(f).g+s_i(f).D_i(g)
$$
for any $f,g\in\cO(R)$. One may accordingly take the smash product of the nil Hecke algebra with $\cO(R)$. The resulting algebra, denoted $\cH$, is free over its subalgebra $\cO(R)$ with respect to both left and right multiplication, with basis $\{D_w\}_{w\in \Gamma}$ and is generated by the nil Hecke algebra and $\cO(R)$ subject to the relations:
$$
D_i.f=D_i(f)+s_i(f).D_i.
$$
In fact it suffices to impose these relations when $f=\theta$ is a linear function, in which case the relations become:
$$
D_i.\theta=\langle\alpha_i^\vee,\theta\rangle+s_i(\theta).D_i
$$
where $\alpha_i^\vee$ is the coroot dual to $\alpha_i$. We extend the grading on the nil Hecke algebra to one on $\cH$ by putting the linear functions in degree $1$.

\subsection{}The algebra $\cH$ may also be understood as the largest subalgebra of the smash product $\Gamma\#Frac(\cO(R))$ which stabilizes $\cO(R)$ in its action on $Frac(\cO(R))$; $D_i$ corresponds to the element $\alpha_i^{-1}.(1-s_i)$. Also $\cH$ receives an $\cO(R)$-algebra map from $\Gamma\#\cO(R)$ determined by sending $s_i$ to $1-\alpha_i.D_i$, in a diagram
$$
\Gamma\#\cO(R)\to\cH\to\Gamma\#Frac(\cO(R)).
$$

\subsection{}The main point about $\cH$ is that it is in an appropriate sense dual to the Hopf algebroid $\cO(R\times_{R//\Gamma}R)$, whose comodules are equivalent to quasicoherent sheaves on $R//\Gamma$ since $R\to R//\Gamma$ is faithfully flat. We therefore see that:
\begin{lem}\label{lem1}The $\Gamma$-equivariant quasicoherent sheaf $M$ on $R$ descends to $R//\Gamma$ if and only if the $\Gamma\#\cO(R)$-module structure on $M$ may be extended to an $\cH$-module structure.\end{lem}

\subsection{}\label{sheafify}Suppose $X$ is a scheme over $\bC$. Since $R$ is an irreducible representation of $W$, the category of $W$-equivariant $R$-bundles over $X$ is equivalent to the category of $\bG_m$-torsors over $X$. Recall that $\cH$ is graded, extending the grading of $\cO(R)$ which comes from the dilation of $R$. Therefore given a $W$-equivariant $R$-torsor $\pi:Y\to X$ one may take the underlying $\bG_m$-torsor $\pi:\cL\to X$ and form 
$$
\cH(Y):=\pi_*\cO(\cL)\otimes^{\bG_m}\cH.
$$
This is a quasicoherent sheaf of Hopf algebroids over $\pi_*\cO(Y)$ on $X$ whose fibers are copies of $\cH$ and which receives a natural map from $\Gamma\#\cO(Y):=\pi_*\cO(\cL)\otimes^{\bG_m}(\Gamma\#\cO(R))$. It is dual over $\pi_*\cO(Y)$ to the Hopf algebroid
$$
\pi_*\cO(Y\times_{Y//\Gamma}Y)
$$
whose comodules over $\pi_*\cO(Y)$ are equivalent to quasicoherent sheaves on $Y//\Gamma$, since $Y\to Y//\Gamma$ is faithfully flat. Thus one obtains the generalization of Lemma \ref{lem1}:
\begin{lem}\label{lem2}The $\Gamma$-equivariant quasicoherent sheaf $M$ on $Y$ descends to $Y//\Gamma$ if and only if the $\Gamma\#\cO(Y)$-module structure on $\pi_*M$ may be extended to an $\cH(Y)$-module structure.\end{lem}

\subsection{The Demazure descent algebra.}Now fix a simple reflection $s_i$ and let $M$ be an $\langle s_i\rangle$-equivariant quasicoherent sheaf on $R$. Then $M$ descends to $R//\langle s_i\rangle$ if and only if for every $m\in M$ there is a unique element $m'\in M$ such that $\alpha_i.m'=(1-s_i)(m)$. In that case one may define the operator
$$
\begin{matrix}G_i:&M &\to& M\\
&m&\mapsto&m'.\end{matrix}
$$
By the uniqueness of $m'$, one sees that $G_i$ is linear and satisfies the relations
$$
G_i^2=0
$$
and
$$
G_i(f.m)=D_i(f).m+s_i(f).G_i(m)
$$
for any $f\in\cO(R)$, $m\in M$. This leads us to define the \emph{Demazure descent algebra}, $\cD$, to be the algebra generated by $\cO(R)$ and the symbols $\{G_i\}_{i\in\Sigma}$ subject to the relations
$$
G_i^2=0
$$
and
$$
G_i.f=D_i(f)+s_i(f).G_i
$$
for any $f\in\cO(R)$. As for $\cH$ the second relation follows from the relation
$$
G_i.\theta=\langle\alpha_i^\vee,\theta\rangle+s_i(\theta).G_i
$$
for any linear function $\theta\in\cO(R)$. It is easy to see that $\cD$ is free over its subalgebra $\cO(R)$ with respect to both left and right multiplication, with basis consisting of all words in $\{G_i\}_{i\in\Sigma}$ without double letters.

\subsection{}A $\cD$-module is precisely the same thing as a quasicoherent sheaf on $R$ which descends to $R//\langle s_i\rangle$ for each $i\in\Sigma$. By Lemma \ref{lem1}, a quasicoherent sheaf on $R$ which descends to $R//\Gamma$ is the same thing as a module for the quotient
$$
\cH=\cD/(\cD B_D\cD)
$$
where $B_D$ is the set of Demazure braid relations, 
$$
B_D=\{\underbrace{G_iG_jG_i\ldots}_{m_{i,j}}-\underbrace{G_jG_iG_j\ldots}_{m_{i,j}}\}_{i,j\in\Gamma}.
$$
On the other hand, a $\Gamma$-equivariant quasicoherent sheaf on $R$ which descends to $R//\langle s_i\rangle$ for every simple reflection $s_i$ is the same thing as a module for the quotient
$$
\cI:=\cD/(\cD B\cD)
$$
where $B$ is the set of Coxeter braid relations,
$$
B=\{\underbrace{s_is_js_i\ldots}_{m_{i,j}}-\underbrace{s_js_is_j\ldots}_{m_{i,j}}\}_{i,j\in\Gamma}
$$
where we have set $s_i=1-\alpha_i.G_i\in\cD$ for each $i\in\Sigma$.

\subsection{}Note that $\cD$ is graded, with each $G_i$ in degree $-1$ and each linear operator $\theta$ on $R$ in degree $1$, so that the quotient map 
$$
\cD\to\cH=\cD/(\cD B_D\cD)
$$
respects the grading. Since the generators of $B$ are also homogeneous, $\cD/(\cD B\cD)$ is also graded and the quotient map
$$
\cD\to\cI=\cD/(\cD B\cD)
$$
respects the grading. As in \ref{sheafify}, for a $W$-equivariant $R$-bundle $\pi:Y\to X$, we obtain the sheaves of algebras $\cD(Y)$, $\cI(Y)$ on $X$. They both receive algebra maps from $\pi_*\cO(Y)$, and are locally free on both sides over $\pi_*\cO(Y)$, and we have the sequence of surjective algebra homomorphisms:
$$
\cD(Y)\to\cI(Y)\to\cH(Y)
$$
which affine locally is identified up to $\bG_m$-action with $\cD\to\cI\to\cH$. A $\cD(Y)$-module over $\pi_*\cO(Y)$ is the same thing as a quasicoherent sheaf on $Y$ which descends to $Y//\langle s_i\rangle$ for each simple refection $s_i$. An $\cI(Y)$-module over $\pi_*\cO(Y)$ is the same thing as a $\Gamma$-equivariant quasicoherent sheaf which descends to $Y//\langle s_i\rangle$ for each simple refection $s_i$. Therefore to prove Theorem \ref{thm1}, it suffices to prove that the natural map $\cI(Y)\to\cH(Y)$ is an isomorphism. For this it is enough to prove the case $Y=R$.

\section{Proof}

\subsection{}Recall we have the projection map $\cI\to\cH$, so that $\cD B\cD\subset \cD B_D\cD$. Let
$$
B_{k,l}=\underbrace{s_ks_ls_k\ldots}_{m_{k,l}}-\underbrace{s_ls_ks_l\ldots}_{m_{k,l}}
$$
be one of the elements of $B$. Recall that $s_i=1-\alpha_iG_i$ for each $i$. Expanding $B_{k,l}$ with respect to the basis as a left $\cO(R)$-module consisting of monomials in the symbols $G_i$, we get 
$$
\begin{matrix*}[l]B_{k,l}&=&(-1)^{m_{k,l}}(\underbrace{\alpha_k.s_k(\alpha_l).s_ks_l(\alpha_k)\ldots}_{m_{k,l}}\underbrace{G_kG_lG_k\ldots}_{m_{k,l}}-\underbrace{\alpha_l.s_l(\alpha_k).s_ls_k(\alpha_l)\ldots}_{m_{k,l}}\underbrace{G_lG_kG_l\ldots}_{m_{k,l}})\\
&&+ l.o.t.\end{matrix*}
$$
Here the lower order terms are left $\cO(R)$-linear combinations of double-letter-free words in $G_k$, $G_l$ of length strictly less than $m_{k,l}$. Such words correspond to elements of $\Gamma$ with unique reduced expressions. Therefore in $\cH$, the images of such words, together with the common image of $\underbrace{G_kG_lG_k\ldots}_{m_{k,l}}$ and $\underbrace{G_kG_lG_k\ldots}_{m_{k,l}}$, are $\cO(R)$-linearly independent. It follows that the lower order terms are zero, and also that the two coefficients are equal:
$$
\underbrace{\alpha_k.s_k(\alpha_l).s_ks_l(\alpha_k)\ldots}_{m_{k,l}}=\underbrace{\alpha_l.s_l(\alpha_k).s_ls_k(\alpha_l)\ldots}_{m_{k,l}}.
$$
We are able to compute them exactly. Let us embed the rank two root system $X_{k,l}$ consisting of those roots in the span of $\alpha_k,\alpha_l$ inside two-dimensional real Euclidean space in the usual manner, so that the angle measured clockwise from $\alpha_k$ to $\alpha_l$ is less than $\pi$. Then the positive roots in $X_{k,l}$ are precisely those which lie in the $\R^+$-cone swept out by rotating the half line $\R^+\alpha_k$ clockwise as far as $\R^+\alpha_l$. Then $s_k\alpha_l$ is the root in $X_{k,l}$ closest in angle clockwise from $\alpha_k$, and the rotation $s_ks_l$ sends any root $\alpha$ in $X_{k,l}$ to the root second closest in angle clockwise from $\alpha$. It follows that the roots
$$
\underbrace{\alpha_k, s_k(\alpha_l), s_ks_l(\alpha_k),\ldots}_{m_{k,l}}
$$
are precisely the positive roots of $X_{k,l}$, ordered clockwise from $\alpha_k$ to $\alpha_l$. Likewise the roots
$$
\underbrace{\alpha_l, s_l(\alpha_k), s_ls_k(\alpha_l),\ldots}_{m_{k,l}}
$$
are precisely the positive roots of $X_{k,l}$ ordered in the other direction. We may therefore write
$$
B_{k,l}=(-1)^{m_{k,l}}\Delta_{k,l}B^D_{k,l}
$$
where $\Delta_{k,l}$ is the product of positive roots in $X_{k,l}$ and $B^D_{k,l}=\underbrace{G_kG_lG_k\ldots}_{m_{k,l}}-\underbrace{G_lG_kG_l\ldots}_{m_{k,l}}$.

\subsection{}We are required to prove that $B^D_{k,l}\in \cD B \cD$.

\begin{lem}\label{lem3}For every integer $n>0$ and every word $\Xi$ in the symbols $D_k$, $D_l$, the elements $\Xi(\Delta_{k,l}).\underbrace{G_kG_lG_k\ldots}_{m_{k,l}+n}$ and $\Xi(\Delta_{k,l}).\underbrace{G_lG_kG_l\ldots}_{m_{k,l}+n}$ are both contained in $\cD B \cD$.\end{lem}
\begin{proof}For brevity we set
$$
A_k^n:=\underbrace{G_kG_lG_k\ldots}_{m_{k,l}+n}
$$
and
$$
A_l^n:=\underbrace{G_lG_kG_l\ldots}_{m_{k,l}+n}.
$$
We proceed by induction on the length of the word $\Xi$. The case of length $0$ follows by multiplying $\pm B_{k,l}$ on the right by words in $G_k$, $G_l$. Suppose next that the claim is known for all $\Xi$ of length at most $p$. Let $\Xi$ be a word of length $p+1$, say $\Xi=D_k\Xi'$ for some word $\Xi'$ of length $p$ (the proof is the same if $\Xi$ starts with $D_l$). By hypothesis,
$$
\Xi'(\Delta_{k,l}).A_k^n
$$
and
$$
\Xi'(\Delta_{k,l}).A_l^n
$$
are both contained in $\cD B \cD$. Multiplying the first element on the left by $G_k$, we get that
$$
\begin{matrix*}[l]
G_k.\Xi'(\Delta_{k,l}).A_k^n&=&\Xi(\Delta_{k,l}).A_k^n+s_k(\Xi'(\Delta_{k,l})).G_k.A_k^n\\
&=&\Xi(\Delta_{k,l}).A_k^n
\end{matrix*}
$$
is contained in $\cD B\cD$. Multiplying the second element on the left by $G_k$, we get that
$$
\begin{matrix*}[l]
G_k.\Xi'(\Delta_{k,l}).A_l^n&=&\Xi(\Delta_{k,l}).A_l^n+s_k(\Xi'(\Delta_{k,l})).G_k.A_l^n\\
&=&\Xi(\Delta_{k,l}).A_k^n+s_k(\Xi'(\Delta_{k,l})).A_k^{n+1}\\
&=&\Xi(\Delta_{k,l}).A_k^n+\Xi'(\Delta_{k,l}).A_k^{n+1}-\alpha_k.\Xi(\Delta_{k,l}).A_k^{n+1}
\end{matrix*}
$$
is contained in $\cD B\cD$. The second term of the RHS is contained in $\cD B\cD$ by hypothesis, while the third term of the RHS is contained in $\cD B\cD$, as one sees from the previous equation (substituting $n+1$ for $n$ and multiplying $\alpha_k$). Thus $\Xi(\Delta_{k,l}).A_k^n$ is contained in $\cD B\cD$, as required.\end{proof}

\begin{lem}\label{handy}There exists an element $Z$ of the left $\bC[\alpha_k,\alpha_l]$-submodule of $\cH$ spanned by the words in $D_k$, $D_l$ such that $Z(\Delta_{k,l})=1$.\end{lem}
\begin{proof}We note that the submodule in question is the subalgebra generated by $\alpha_k$, $\alpha_l$, $D_k$ and $D_l$, which is the algebra associated to the rank $2$ root system generated by $\alpha_k$, $\alpha_l$ in the same way that $\cH$ is associated to the root system of $\Gamma$. We will call this algebra $\cH_{k,l}$ and write $\Gamma_{k,l}$ for its associated Coxeter group. As we have already remarked, $\cH_{k,l}$ is equal to the subalgebra of $\Gamma_{k,l}\#\bC(\alpha_k,\alpha_l)$ consisting of all those elements which send $\bC[\alpha_k,\alpha_l]$ to itself in the natural action on $\bC(\alpha_k,\alpha_l)$. We note that
$$
\Delta_{k,l}^{-1}\sum_{g\in\Gamma_{k,l}}sgn(g)g
$$
is such an element, and it sends $\Delta_{k,l}$ to the non-zero scalar $|\Gamma_{k,l}|$.\end{proof}

In combination, we obtain:
\begin{lem}\label{yes}The elements $\underbrace{G_kG_lG_k\ldots}_{m_{k,l}+1}$ and $\underbrace{G_lG_kG_l\ldots}_{m_{k,l}+1}$ are both contained in $\cD B \cD$.\end{lem}

Finally, we have:
\begin{thm}$B^D_{k,l}\in \cD B \cD$.\end{thm}
\begin{proof}By Lemma \ref{handy}, it suffices to show that for every word $\Xi$ in the symbols $D_k$, $D_l$, the element $\Xi(\Delta_{k,l}).B^D_{k,l}$ is contained in $\cD B \cD$. Again we proceed by induction on the length of $\Xi$. The length $0$ case is the fact that $B_{k,l}\in\cD B \cD$. Suppose next that the claim is known for all $\Xi$ of length at most $p$. Let $\Xi$ be a word of length $p+1$, say $\Xi=D_k\Xi'$ for some word $\Xi'$ of length $p$ (the proof is the same if $\Xi$ starts with $D_l$). By hypothesis, $\Xi'(\Delta_{k,l}).B^D_{k,l}$ is contained in $\cD B \cD$. Therefore
$$
\begin{matrix}
G_k.\Xi'(\Delta_{k,l}).B^D_{k,l}&=&\Xi(\Delta_{k,l}).B^D_{k,l}+s_k(\Xi'(\Delta_{k,l})).G_k.B^D_{k,l}\\
&=&\Xi(\Delta_{k,l}).B^D_{k,l}-s_k(\Xi'(\Delta_{k,l})).\underbrace{G_kG_lG_k\ldots}_{m_{k,l}+1}
\end{matrix}
$$
is contained in $\cD B \cD$. By Lemma \ref{yes} we are done.
\end{proof}

\

\end{document}